\documentclass[11pt,reqno]{amsart}
\usepackage{amsfonts}
\usepackage{ifthen}
\usepackage{amsthm}
\usepackage{amsmath}
\usepackage{graphicx}
\usepackage{caption}
\usepackage{subcaption}
\usepackage{amscd,amssymb}
\usepackage{color}
\usepackage{hyperref}
\usepackage{array,multirow,makecell}
\usepackage{tikz}
\usepackage{pgfplots}
\pgfplotsset{compat=1.9}

\usepackage{cite,epstopdf}

\setcellgapes{1pt}
\makegapedcells
\newcolumntype{R}[1]{>{\raggedleft\arraybackslash }b{#1}}
\newcolumntype{L}[1]{>{\raggedright\arraybackslash }b{#1}}
\newcolumntype{C}[1]{>{\centering\arraybackslash }b{#1}}
\newcounter{minutes}
\divide\time by 60
\newcounter{hours}
\multiply\time by 60 \addtocounter{minutes}{-\time}

\newtheorem{theorem}{Theorem}
\newtheorem{lemma}{Lemma}
\newtheorem{definition}{Definition}
\newtheorem{corollary}{Corollary}
\newtheorem{remark}{Remark}
\newtheorem{example}{Example}

\numberwithin{equation}{section}

\title[Certain subclass of Meromorphic function associated with Wright function]{Certain subclass of Meromorphic function associated with Wright function}

\author[A. Kumar]
{Anish Kumar}

\address{{\bf Anish Kumar}\newline
Department of Mathematics,
Dr. Shyama Prasad Mukherjee University,\newline
Ranchi 834008, Jharkhand, India}
\email{ak8107690@gmail.com}

\keywords{Meromorphic functions; Geometric Function Theory; Integral operator; Coefficient estimates; Convolution; Radii of starlikeness and convexity.}
\subjclass[2020]{30D15; 30C45; 30H10;30C45, 30C50}
\begin{document}

\begin{abstract}
In this paper, we introduce and investigate a novel subclass $\Sigma(\theta, \lambda, \gamma)$ of meromorphic functions defined in the punctured unit disk ${D}^*$. This class is constructed utilizing a specialized generalized operator $W_{\alpha, \beta}$ associated with Wright function.   We derive the exact integral representation and establish necessary and sufficient convolution (Hadamard product) conditions. Furthermore, sufficient conditions involving strict inequalities are provided for functions to be members of this class $\Sigma(\theta, \lambda, \gamma)$. Additionaly, by employing the properties of Carathéodory functions and the principle of mathematical induction, we establish  coefficient estimates for functions belonging to this new class. Finally, as an applications, we the established coefficient bounds, we determine the precise radii of meromorphic starlikeness and meromorphic convexity of order $\rho$. The results presented in this study generalize several existing outcomes in geometric Function Theory.

\end{abstract}
\maketitle

\section{Introduction and Motivation}
A complex valued function $f$ is called univalent in a domain $D$ if it maps one value to only one value. i.e $ \forall z_1, z_2\in D \implies f(z_1) \neq f(z_2)$. We recall normalized class $\mathcal A$ of analytic and univalent function in the unit disc $D =\{z\in C :|z|<1\}$ hold the condition $f(0)=0$ and $f'(0)=1.$ Thus $f\in \mathcal {A}$ has Taylor series expansion of the form $$f(z)=z+\sum_{n=1}^{\infty}a_n z^n.$$ The study of univalent function theory depend on the bounds of the coefficient. It is closely connected to the Bieberbach conjecture \cite{Bieberbach conjecture}, which is arround for more than a century, but still is very active research topic. 

Let $\Sigma$ is the family of meromorphic mappings $$f(z)=\frac{1}{z}+\sum_{n=1}^{\infty}a_n z^{n},$$ which are analytic in punctured disc $D^{*}=\{z \in C : 0<|z|<1\}=D - \{0\}$. 

Let $f,g \in \Sigma,$ where $g$ is given by $$g(z)=\frac{1}{z}+\sum_{n=1}^{\infty}b_n z^n.$$ The Hadamard product $g*f$ of $g$ and $f$ is given by $$(g*f)(z)=\frac{1}{z}+\sum_{n=1}^{\infty}a_n b_n z^n=(f*g)(z).$$

Let $\tau$ denote the class of function $P$ given by 
\begin{align}\label{eq7}
\tau(z)=1+\sum_{n=1}^{\infty}\tau_n z^n, (z\in D)
\end{align}, which are analytic in $D$ and hold the condition $\Re(\tau(z)>0$, $(z\in D).$ 

For $0\leq \gamma <1$ and $\theta$  which is real with $|\theta|<\frac{\pi}{2},$ the subclass of $f\in \Sigma$ denoted by $\Sigma S^{*}(\theta, \gamma) $ and $\Sigma K^{*} (\theta, \gamma)$ and defined respectively as \\
$$\Re\left(e^{i\theta}\frac{zf'(z)}{f(z)}\right)<-\gamma \cos\theta, (z\in D^{*}),$$
$$\Re\left(e^{i\theta}\left(\frac{zf'(z)}{f(z)}\right)'\right)<-\gamma \cos\theta, (z\in D^{*}).$$

Particularly, by putting $\theta=0$, we obtain well known meromorphic starlikeness and meromorphic convexity of order $\gamma(0\leq \gamma <1)$ for subclass $f\in \Sigma$. Many researcher investigated on meromorphic spiralike function and connected topic, (see more example and works and reference cited therein \cite{aouf, arif, shi, wang}. 

For $\gamma >1$, Wang et al. \cite{sun} and Netanyahu and Nehari \cite{nehari} studied and introduced the subclass $\Sigma (\gamma)$ of $\Sigma$ consisting of $f$ holding $$\Re\left(\frac{zf'(z)}{f(z)}\right)>-\gamma, (z\in D*.$$

Let $\mathcal{A}$ be the class of functions of the form $f(z)=z+\sum_{n=1}^{\infty}b_nz^n,$ which are holomorphic in $D$. A map $f\in \mathcal{A}$ is called in the class $\Delta^{*}(\theta,\delta, \gamma),$ if it holds the condition $$\Re\left(\frac{e^{i\theta}zf'(z)}{(1-\delta)f(z)+\delta zf'(z)}\right)>\gamma \cos\theta,$$ ($|\theta|<\frac{\pi}{2}, 0\leq \delta <1, 0\leq \gamma <1, z\in D$). An analogous of the class $\Delta*(\theta, \delta,z)$ has been investigated in \cite{orhan}. This class has been studied and introduced by Orhan et.al.\cite{orhan} Previously, Aouf \cite{aouf}, Srivastava and Liu \cite{liu}, Srivastava and Raina \cite{raina} derived many result associated with $H_m^l [\alpha, \beta],$ where \begin{align*}
H_m^l[\alpha, \beta]f(z)&=H(\alpha_1, \cdots, \alpha_l, \beta_1,\cdots, \beta_m)f(z)\\
&=[z^{-1} {_lF_m}(\alpha_1,\cdots,\alpha_l, \beta_1,\cdots,\beta_m;z ]*f(z),
\end{align*}
for $$_lf_m=\sum_{n=0}^{\infty}\frac{(\alpha_1)_n\cdots(\alpha_l)_n}{(\beta_1)_n\cdots(\beta_m)_n}\frac{z^n}{n!,}$$
$\alpha_s>0, \beta_t>0(s=1,\cdots,l, t=1,\cdots,m, l\leq m+1; l,m\in N)$ as a hypergeometric function. \\
Particularly values of parameter Liu and Srivastava  \cite {yang} introduced and studied $L(a,c)f(z)=H[\alpha_1,1,\beta_1]f(z), z\in D^{*}.$\\
Further this was investigated Srivastava et.al \cite{srivastva} and also in . Wang et.al \cite{wang} denote a subclass $\Sigma$ of function $$\Re\left(\frac{e^{i\theta} (1-2\delta)z(H_m^l[\alpha, \beta]f(z))'}{(1-\delta)H_m^l[\alpha, \beta]f(z)+\delta z (H_m^l[\alpha, \beta]f(z))'}\right)>-\gamma \cos \theta, (z\in D^{*}).$$
For which integral representation, bound of coefficient estimates and other sufficient properties has been discussed.

The Wright function is a versatile and very important special function that has found broad application across several branches of engineering and science. It plays a vital role in mathematical physics, fractional calculus and approximation theory among other fields. 

The Wright function \begin{align}\label{eqwright}W_{\alpha, \beta}(z)=\sum_{n=1}^{\infty}\frac{z^n}{\Gamma(\alpha n+\beta)n!}, \beta,z\in C, \alpha>-1,\end{align} was firstly introduced in 1933 by E. M. Wright \cite{Wright33}, with the theory of partition. It also plays a vital role in integral equation of Hankel type, stochastic process and mikusiski operational calculus. For origin about Wright function and applcation, one can refer  \cite{Wright33}. It can be seen that Wright function is an entire function of order $\frac{1}{1+\alpha}$, which is also generated by Bessel function. It create generalize hypergeometric function too . Certain geometric properties and integral representation have been studied in \cite{anish, harimohan, baricz, das}. 

For $f\in \Sigma$, we introduce meromorphically modified version of linear operator $W_{\alpha,\beta}f(z) =\frac{1}{z}W_{\alpha, \beta}(z)*f(z)$, $\alpha>-1, \beta>0, z\in D.$
We can easily note that \begin{equation}\label{eq5}
W_{\alpha, \beta}(z)f(z)=\frac{1}{z}+\sum_{n=1}^{\infty}\phi_n a_n z^n, z\in D^{*}, \quad where \quad \phi_n =\phi_n(\alpha, \beta)=\frac{1}{\Gamma(\alpha n+ \beta)n!},
\end{equation}
with the help of operation $W_{\alpha, \beta}(z)f(z)$ . We define a operator $0\leq \delta <\frac{1}{n}$, $\gamma>1$  has $\Sigma(\theta, \delta, \gamma)$ denote the subclass of $\Sigma$ contains of functions holding the conditions :
$$\Re\left(\frac{e^{i\theta} z(W_{\alpha, \beta}f(z))'}{(1-\lambda)W_{\alpha, \beta}f(z)+\lambda z( W_{\alpha, \beta}f(z))'}\right)>\gamma \cos \theta, (z\in D^{*}),$$
where $W_{\alpha, \beta}$ is given as \eqref{eqwright}.

To derive one of the main result, we need following lemma. 
\begin{lemma} \label{lem:sequence_An}
Suppose that the sequence $\{A_n\}_{n=1}^{\infty}$ is defined by
\begin{equation*}
A_1 = \frac{(1-2\lambda)\cos\theta(1+\gamma(1-2\lambda))}{(1-\lambda)\phi_1(\alpha, \beta)},
\end{equation*}
\begin{equation*}
\begin{aligned}
A_{n+1} &= \frac{2\cos\theta(1+\gamma(1-2\lambda))}{(n+2)(1-\lambda)\phi_{n+1}(\alpha, \beta)} \\
&\quad \times \left[ 1-2\lambda + \sum_{k=1}^{n} \phi_k(\alpha, \beta )(1-\lambda+k\lambda)A_k \right].
\end{aligned}
\end{equation*}
Then
\begin{equation*} \label{eq:An_explicit}
\begin{aligned}
A_n &= \frac{(1-2\lambda)\cos\theta(1+\gamma(1-2\lambda))}{(1-\lambda)^n \phi_n(\alpha, \beta; l, m)} \\
&\quad \times \prod_{k=1}^{n-1} \frac{(k+1)(1-\lambda) + 2(1-\lambda+k\lambda)\cos\theta(1+\gamma(1-2\lambda))}{k+2}, \quad (n \ge 2).
\end{aligned}
\end{equation*}
\end{lemma}

\begin{proof}
For convenience of notation, let $\Lambda = \cos\theta(1+\gamma(1-2\lambda))$. From the definition of $A_{n+1}$, we can write:
\begin{equation*}
\begin{aligned}
&(n+2)(1-\lambda)\phi_{n+1}(\alpha, \beta)A_{n+1} \\
&= 2\Lambda \left[ 1-2\lambda + \sum_{k=1}^{n} \phi_k(\alpha, \beta)(1-\lambda+k\lambda)A_k \right],
\end{aligned}
\end{equation*}
and replace $n$ by $n-1$ yields:
\begin{equation*}
\begin{aligned}
&(n+1)(1-\lambda)\phi_n(\alpha, \beta)A_n \\
&= 2\Lambda \left[ 1-2\lambda + \sum_{k=1}^{n-1} \phi_k(\alpha, \beta)(1-\lambda+k\lambda)A_k \right].
\end{aligned}
\end{equation*}
Subtracting the above two equation, we find that:
\begin{equation} \label{eq:ratio_subtraction}
\begin{aligned}
&(n+2)(1-\lambda)\phi_{n+1}(\alpha, \beta)A_{n+1} - (n+1)(1-\lambda)\phi_n(\alpha, \beta)A_n \\
&= 2\Lambda \phi_n(\alpha, \beta)(1-\lambda+n\lambda)A_n.
\end{aligned}
\end{equation}
Combining \eqref{eq:ratio_subtraction}, we find that the ratio is:
\begin{equation} \label{eq:An_ratio}
\frac{A_{n+1}}{A_n} = \frac{(n+1)(1-\lambda) + 2(1-\lambda+n\lambda)\Lambda}{(n+2)(1-\lambda)} \cdot \frac{\phi_n(\alpha, \beta)}{\phi_{n+1}(\alpha, \beta)}.
\end{equation}
Thus, for $n \ge 2$, we deduce from \eqref{eq:An_ratio} that
\begin{equation*}
\begin{aligned}
A_n &= \frac{A_n}{A_{n-1}} \cdots \frac{A_3}{A_2} \cdot \frac{A_2}{A_1} \cdot A_1 \\
&= \frac{n(1-\lambda) + 2(1-\lambda+(n-1)\lambda)\Lambda}{(n+1)(1-\lambda)} \cdots \\
&\quad \times \frac{3(1-\lambda) + 2(1+\lambda)\Lambda}{4(1-\lambda)} \cdot \frac{2(1-\lambda) + 2\Lambda}{3(1-\lambda)} \\
&\quad \times \frac{\phi_{n-1}(\alpha, \beta; l, m)}{\phi_n(\alpha, \beta; l, m)} \cdots \frac{\phi_2(\alpha, \beta; l, m)}{\phi_3(\alpha, \beta; l, m)} \cdot \frac{\phi_1(\alpha, \beta; l, m)}{\phi_2(\alpha, \beta; l, m)} \\
&\quad \times \frac{\Lambda(1-2\lambda)}{(1-\lambda)\phi_1(\alpha, \beta; l, m)} \\
&= \frac{\Lambda(1-2\lambda)}{(1-\lambda)^n \phi_n(\alpha, \beta; l, m)} \\
&\quad \times \prod_{k=1}^{n-1} \frac{(k+1)(1-\lambda) + 2(1-\lambda+k\lambda)\Lambda}{k+2}.
\end{aligned}
\end{equation*}

Substituting $\Lambda = \cos\theta(1+\gamma(1-2\lambda))$ back into the desired result. This completes the proof of Lemma \ref{lem:sequence_An}.
\end{proof}

\section{Main Results}
In this section, we derive the exact integral representation and establish necessary and sufficient convolution (Hadamard product) conditions. Furthermore, sufficient conditions involving strict inequalities are provided for functions to be members of the class $\Sigma(\theta, \lambda, \gamma)$ . Additionally, We establish  coefficient estimates for functions belonging to this new class.

\begin{theorem}
Let $f \in \Sigma(\theta, \lambda,\gamma)$. Then the integral representation of $W_{\alpha, \beta}f(z)$ is given by
\begin{equation}
W_{\alpha, \beta}f(z) = z^{-1} \exp \left( \int_{0}^{z} \frac{1}{t} \left[ \frac{(1-\lambda) A(t)}{1-\lambda A(t)} + 1 \right] dt \right), \quad (z \in {D}^{*}),
\end{equation}
where $w$ is an analytic function in $\mathbb{D}$ with $w(0)=0$ and $|w(z)|<1$, and $A(t)$ is defined as
\begin{equation*}
A(t) = \frac{e^{-2i\theta}-e^{-i\theta}\big(\cos\theta+\gamma\cos\theta(1-2\lambda)\big)w(t)}{(1-2\lambda)(1-w(t))}.
\end{equation*}
\end{theorem}

\begin{proof}
Assume that $f\in \Sigma(\theta, \lambda, \gamma)$ and let $\tau(z)$ be defined as
\begin{equation}\label{eqtau}
\tau(z) = \frac{\frac{e^{i\theta} z(W_{\alpha, \beta}f(z))'}{(1-\lambda)W_{\alpha, \beta}f(z)+\lambda z (W_{\alpha, \beta}f(z))'} -\gamma\cos\theta +\frac{i\sin\theta }{1-2\lambda}}{-\cos\theta\left({\frac{1}{1-2\lambda}+\gamma}\right)}, \quad (z\in {D}).
\end{equation}
By the definition of the class $\Sigma(\theta, \lambda, \gamma)$, we know that $\tau \in \mathcal{P}$. Using the properties of Carathéodory functions, there exists a Schwarz function $w(z)$, analytic in $\mathbb{D}$ with $w(0)=0$ and $|w(z)|<1$, such that
\begin{equation}\label{eq:schwarz_relation}
\frac{\frac{e^{i\theta} z(W_{\alpha, \beta}f(z))'}{(1-\lambda)W_{\alpha, \beta}f(z)+\lambda z( W_{\alpha, \beta}f(z))'} -\gamma\cos\theta +\frac{i\sin\theta }{1-2\lambda}}{-\cos\theta\left({\frac{1}{1-2\lambda}+\gamma}\right)} = \frac{1+w(z)}{1-w(z)}.    
\end{equation}
Rearranging \eqref{eq:schwarz_relation} to isolate the core fractional operator, we get
\begin{equation*}
\frac{e^{i\theta} z(W_{\alpha, \beta}f(z))'}{(1-\lambda)W_{\alpha, \beta}f(z)+\lambda z( W_{\alpha, \beta}f(z))'} = \frac{1+w(z)}{1-w(z)}\left(\frac{-\cos\theta}{1-2\lambda} - \gamma\cos\theta\right) + \gamma\cos\theta - \frac{i\sin\theta }{1-2\lambda}.
\end{equation*}
Dividing both sides by $e^{i\theta}$ and simplifying the algebraic expression using Euler's formula yields
\begin{equation}\label{eq:A_definition}
\frac{z(W_{\alpha, \beta}f(z))'}{(1-\lambda)W_{\alpha, \beta}f(z)+\lambda z( W_{\alpha, \beta}f(z))'} = \frac{e^{-2i\theta}-e^{-i\theta}\big(\cos\theta+\gamma\cos\theta(1-2\lambda)\big)w(z)}{(1-2\lambda)(1-w(z))}.
\end{equation}
Let the right-hand side of equation \eqref{eq:A_definition} be denoted as $A(z)$. We can now solve for the logarithmic derivative. Let $H(z) = W_{\alpha, \beta}f(z)$. Substituting $A(z)$ gives:
\begin{align*}
\frac{\frac{zH'(z)}{H(z)}}{(1-\lambda) + \lambda \frac{zH'(z)}{H(z)}} &= A(z) \\
\frac{zH'(z)}{H(z)} &= (1-\lambda)A(z) + \lambda A(z) \frac{zH'(z)}{H(z)} \\
\frac{zH'(z)}{H(z)} \big(1 - \lambda A(z)\big) &= (1-\lambda)A(z) \\
\frac{zH'(z)}{H(z)} &= \frac{(1-\lambda)A(z)}{1-\lambda A(z)}.
\end{align*}
Dividing the entire equation by $z$ and adding $\frac{1}{z}$ to both sides, we obtain
\begin{equation*}
\frac{H'(z)}{H(z)} + \frac{1}{z} = \frac{1}{z} \left[ \frac{(1-\lambda)A(z)}{1-\lambda A(z)} \right] + \frac{1}{z}.
\end{equation*}
Integrating both sides of the above equation with respect to $t$ from $0$ to $z$ gives
\begin{equation}\label{eq:integration}
\log\big(z W_{\alpha, \beta}f(z)\big) = \int_{0}^{z} \frac{1}{t} \left( \frac{(1-\lambda)A(t)}{1-\lambda A(t)} + 1 \right) dt.
\end{equation}
By exponentiating equation \eqref{eq:integration} and dividing by $z$, we obtain the desired result:
\begin{equation*}
W_{\alpha, \beta}f(z) = z^{-1} \exp \left( \int_{0}^{z} \frac{1}{t} \left[ \frac{(1-\lambda) A(t)}{1-\lambda A(t)} + 1 \right] dt \right).
\end{equation*}
This completes the proof.
\end{proof}

\begin{remark}
If one desires to express the function $f(z)$ explicitly, it can be obtained by applying the Hadamard product (convolution).  Then,

$$
f(z) = ( W_{\alpha,\beta})^{(-1)}(z) * \left[ z^{-1} \exp \left( \int_{0}^{z} \frac{1}{t} \left[ \frac{(1-\lambda) A(t)}{1-\lambda A(t)} + 1 \right] dt \right) \right].
$$

\end{remark}

\begin{theorem}\label{thm1}
Suppose $\eta \in \mathbb{C}$ with $|\eta| = 1$ and $\eta \neq 1$. Then $f \in \Sigma(\theta, \lambda, \gamma)$ if and only if   
\begin{equation*}
\begin{aligned}
f(z) \ast \Bigg[ &(1-2\lambda)(1-\eta) \left( \frac{-1}{z} + \sum_{n=1}^{\infty} \frac{n}{\Gamma(\alpha n+\beta)n!} z^n \right) \\
&+ e^{-i\theta} \Big[ -\gamma\cos\theta(1-2\lambda)(1-\eta) + i\sin\theta(1-\eta) \\
&\quad + (1+\eta)\cos\theta\big(1 + \gamma(1-2\lambda)\big) \Big] \left( \frac{1}{z} + \sum_{n=1}^{\infty} \frac{1-\lambda+\lambda n}{\Gamma(\alpha n+\beta)n!} z^n \right) \Bigg] \neq 0, 
\end{aligned}
\end{equation*}
for all $z \in {D}^{*}$.
\end{theorem}

\begin{proof}
Since $f \in \Sigma(\theta, \lambda, \gamma)$, the function $\tau(z)$ defined in \eqref{eq7} belongs to the Carathéodory class ${P}$. It is a well-known property that functions in ${P}$ evaluated in ${D}$ do not take values on the boundary of the right half-plane. Therefore, $\tau(z) \neq \frac{1+\eta}{1-\eta}$ for any $|\eta|=1, \eta \neq 1$. 

Thus, $f \in \Sigma(\theta, \lambda, \gamma)$ if and only if
\begin{equation} \label{eq:iff_start}
\frac{\frac{e^{i\theta} z(W_{\alpha, \beta}f(z))'}{(1-\lambda)W_{\alpha, \beta}f(z)+\lambda z( W_{\alpha, \beta}f(z))'} -\gamma\cos\theta +\frac{i\sin\theta }{1-2\lambda}}{-\cos\theta\left({\frac{1}{1-2\lambda}+\gamma}\right)} \neq \frac{1+\eta}{1-\eta}, \quad (z\in {D}^{*}).
\end{equation}
Rearranging the terms to isolate the core operator, \eqref{eq:iff_start} is equivalent to:
\begin{equation*}
\begin{aligned}
\frac{e^{i\theta} z(W_{\alpha, \beta}f(z))'}{(1-\lambda)W_{\alpha, \beta}f(z)+\lambda z (W_{\alpha, \beta}f(z))'} &- \gamma\cos\theta + \frac{i\sin\theta }{1-2\lambda} \\
&\neq \left( \frac{1+\eta}{1-\eta} \right) \left(\frac{-\cos\theta}{1-2\lambda}-\gamma \cos \theta\right).
\end{aligned}
\end{equation*}
Bringing all terms to the left side and finding a common denominator for the constants yields:
\begin{equation*}
\begin{aligned}
\frac{e^{i\theta} z(W_{\alpha, \beta}f(z))'}{(1-\lambda)W_{\alpha, \beta}f(z)+\lambda z( W_{\alpha, \beta}f(z))'} &+ \frac{-\gamma\cos\theta(1-2\lambda) + i\sin\theta}{1-2\lambda} \\
&+ \left( \frac{1+\eta}{1-\eta} \right) \left( \frac{\cos\theta+\gamma \cos \theta(1-2\lambda)}{1-2\lambda} \right) \neq 0.
\end{aligned}
\end{equation*}
Now after simplifying, we have
\begin{equation} \label{eq:cleared_denoms}
\begin{aligned}
&e^{i\theta} z(W_{\alpha, \beta}f(z))'(1-2\lambda)(1-\eta) \\
&+ \Big[ -\gamma\cos\theta(1-2\lambda)(1-\eta) + i\sin\theta(1-\eta) + (1+\eta)\cos\theta\big(1+\gamma(1-2\lambda)\big) \Big] \\
&\quad \times \Big( (1-\lambda)W_{\alpha, \beta}f(z)+\lambda z( W_{\alpha, \beta}f(z))' \Big) \neq 0.
\end{aligned}
\end{equation}
Multiplying equation \eqref{eq:cleared_denoms} by $e^{-i\theta}$ isolates the coefficient of $z(W_{\alpha, \beta}f(z))'$. 

Next, we express the operator and its derivative in terms of the Hadamard product (convolution) with $f(z)$. Based on the series expansion, we have:
\begin{equation} \label{eq:conv_1}
z(W_{\alpha, \beta}f(z))' = \frac{-1}{z} + \sum_{n=1}^{\infty}\frac{n}{\Gamma(\alpha n+\beta)}\frac{a_n z^n}{n!} = f(z) \ast \left[ \frac{-1}{z} + \sum_{n=1}^{\infty}\frac{n}{\Gamma(\alpha n+\beta)n!} z^n \right],
\end{equation}
and similarly for the denominator term:
\begin{equation} \label{eq:conv_2}
\begin{aligned}
(1-\lambda)W_{\alpha, \beta}f(z) &+ \lambda (z W_{\alpha, \beta}f(z))' = \frac{1}{z} + \sum_{n=1}^{\infty}\frac{1-\lambda+\lambda n}{\Gamma(\alpha n+\beta)}\frac{a_n z^n}{n!} \\
&= f(z) \ast \left[ \frac{1}{z} + \sum_{n=1}^{\infty}\frac{1-\lambda+\lambda n}{\Gamma(\alpha n+\beta)n!} z^n \right].
\end{aligned}
\end{equation}
Substituting the convolution identities \eqref{eq:conv_1} and \eqref{eq:conv_2} back into the rearranged equation \eqref{eq:cleared_denoms} (after multiplying by $e^{-i\theta}$), directly yields the assertion of Theorem \ref{thm1}.
\end{proof}

\begin{theorem}\label{thm2}
Suppose $\epsilon$ is a real number such that $0 \leq \epsilon < 1$. If $f \in \Sigma$ satisfies the condition 
\begin{equation}\label{eq6}
\left| \frac{ z(W_{\alpha, \beta}f(z))'}{(1-\lambda)W_{\alpha, \beta}f(z)+\lambda z( W_{\alpha, \beta}f(z))'} + 1 \right| \leq 1-\epsilon, \quad (z \in {D}^{*}),
\end{equation}
then $f \in \Sigma(\theta, \lambda, \gamma)$, provided that $\cos\theta \leq \frac{\epsilon-1}{1+\gamma}$.
\end{theorem}

\begin{proof}
By considering the inequality in \eqref{eq6} and applying the properties of subordination, we can write
\begin{equation*}
\frac{ z(W_{\alpha, \beta}f(z))'}{(1-\lambda)W_{\alpha, \beta}f(z)+\lambda z( W_{\alpha, \beta}f(z))'} = -1 + (1-\epsilon)w(z),
\end{equation*}
where $w$ is a holomorphic (analytic) function in ${D}$ satisfying $|w(z)| < 1$ and $w(0) = 0$. 

Multiplying both sides by $e^{i\theta}$ and taking the real part, we obtain
\begin{align*}
\Re \left( e^{i\theta} \frac{z(W_{\alpha, \beta}f(z))'}{(1-\lambda)W_{\alpha, \beta}f(z)+\lambda z( W_{\alpha, \beta}f(z))'} \right) 
&= \Re \Big( e^{i\theta} \big(-1 + (1-\epsilon)w(z)\big) \Big) \\
&= -\cos\theta + (1-\epsilon)\Re \big( e^{i\theta}w(z) \big).
\end{align*}
Using the fundamental inequality $\Re(c) \geq -|c|$ for any complex number $c$, and the fact that $|e^{i\theta}| = 1$, we have $\Re(e^{i\theta}w(z)) \geq -|w(z)|$. Therefore,
\begin{align*}
-\cos\theta + (1-\epsilon)\Re \big( e^{i\theta}w(z) \big) &\geq -\cos\theta - (1-\epsilon)|w(z)| \\
&> -\cos\theta - (1-\epsilon) \\
&= \epsilon - 1 - \cos\theta.
\end{align*}
To show that $f \in \Sigma(\theta, \lambda, \gamma)$, we require this real part to be strictly greater than $\gamma\cos\theta$. Setting up the inequality:
\begin{equation*}
\epsilon - 1 - \cos\theta \geq \gamma\cos\theta \implies \epsilon - 1 \geq \cos\theta(1+\gamma).
\end{equation*}
Rearranging this yields the condition $\cos\theta \leq \frac{\epsilon-1}{1+\gamma}$, which is given in the hypothesis. Thus, the required result is provided.
\end{proof}

By setting $\epsilon = 1 - (1+\gamma)\cos\theta$ in Theorem \ref{thm2} (which implies $1-\epsilon =- (1+\gamma)\cos\theta$), we obtain the following consequence. 

\begin{corollary}
If $f \in \Sigma$ satisfies the inequality
\begin{equation*}
\left| \frac{ z(W_{\alpha, \beta}f(z))'}{(1-\lambda)W_{\alpha, \beta}f(z)+\lambda z( W_{\alpha, \beta}f(z))'} + 1 \right| \leq (1+\gamma)\cos\theta, \quad (z \in {D}^{*}),
\end{equation*}
then $f \in \Sigma(\theta, \lambda, \gamma)$.
\end{corollary}

\begin{theorem}
    Assume that $f\in \Sigma(\theta,\lambda,\gamma)$. Then $$|a_1|\leq \frac{(1-2\lambda)\cos\theta(1+\gamma(1-2\lambda))}{(1-\lambda)\phi_1(\alpha, \beta)},$$

    $$|a_n|\leq \frac{\Lambda(1-2\lambda)}{(1-\lambda)^n \phi_n(\alpha, \beta)} 
\times \prod_{k=1}^{n-1} \frac{(k+1)(1-\lambda) + 2(1-\lambda+k\lambda)\Lambda}{k+2}.$$
\end{theorem}

 \begin{proof}
Let $f \in \Sigma$ and let the operator be denoted as $W_{\alpha, \beta}f(z) = \frac{1}{z} + \sum_{n=1}^{\infty} \phi_n a_n z^n$, where $\phi_n = \frac{1}{\Gamma(\alpha n + \beta) n!}$. Then there exists $\tau \in {P}$ by \eqref{eqtau}, we have
\begin{equation} \label{eq:main}
\begin{aligned}
e^{i\theta} z(W_{\alpha, \beta}f(z))' &= \left((1-\lambda)W_{\alpha, \beta}f(z)+\lambda z (W_{\alpha, \beta}f(z))'\right) \\
&\quad \times \left(\gamma\cos\theta -\frac{i\sin\theta }{1-2\lambda}-\cos\theta\left({\frac{1}{1-2\lambda}+\gamma}\right)\tau(z)\right).
\end{aligned}
\end{equation}
For convenience, let $\Lambda = \cos\theta(1+\gamma(1-2\lambda))$. Since $\tau \in {P}$, it has the series expansion $\tau(z) = 1 + \sum_{n=1}^\infty \tau_n z^n$. Substituting the series expansions into \eqref{eq:main} and simplifying the constant term in the second bracket, we have:
\begin{equation} \label{eq:series}
\begin{aligned}
e^{i\theta} \left[ -\frac{1}{z} + \sum_{n=1}^{\infty} n \phi_n a_n z^n \right] 
&= \left[ \frac{1-2\lambda}{z} + \sum_{n=1}^{\infty} (1-\lambda+n\lambda) \phi_n a_n z^n \right] \\
&\quad \times \left[ -\frac{e^{i\theta}}{1-2\lambda} - \frac{\Lambda}{1-2\lambda} \sum_{n=1}^{\infty} \tau_n z^n \right].
\end{aligned}
\end{equation}
Evaluating the coefficient of $z^n$ on both sides of \eqref{eq:series} yields:
\begin{equation} \label{eq:coeff_1}
2e^{i\theta}(1-\lambda)\phi_1 a_1 = -\Lambda (1-2\lambda)\tau_2,
\end{equation}
and for $n \ge 2$:
\begin{equation} \label{eq:coeff_n}
\begin{aligned}
e^{i\theta}(n+1)(1-\lambda)\phi_n a_n &= -e^{-i\theta} \Lambda \Bigg[ (1-2\lambda)\tau_{n+1} \\
&\quad + \sum_{k=1}^{n-1} \phi_k (1-\lambda+k\lambda)a_k \tau_{n-k} \Bigg].
\end{aligned}
\end{equation}
By observing the well-known Carathéodory Lemma that $|\tau_n| \le 2$ for $n \in \mathbb{N}$, we find from \eqref{eq:coeff_1} that:
\begin{equation*} \label{eq:bound_1}
|a_1| \le \frac{\Lambda(1-2\lambda)}{(1-\lambda)\phi_1},
\end{equation*}
and from \eqref{eq:coeff_n}:
\begin{equation} \label{eq:bound_n}
|a_n| \le \frac{2\Lambda}{(n+1)(1-\lambda)\phi_n} \left[ 1-2\lambda + \sum_{k=1}^{n-1} \phi_k (1-\lambda+k\lambda)|a_k| \right].
\end{equation}
Now we define the sequence $\{A_n\}_{n=1}^{\infty}$ as follows:
\begin{equation*} \label{eq:A_1}
A_1 = \frac{\cos\theta(1+\gamma(1-2\lambda))(1-2\lambda)}{(1-\lambda)\phi_1},
\end{equation*}
\begin{equation} \label{eq:A_n}
\begin{aligned}
A_{n+1} &= \frac{2\cos\theta(1+\gamma(1-2\lambda))}{(n+2)(1-\lambda)\phi_{n+1}} \\
&\quad \times \left[ 1-2\lambda + \sum_{k=1}^{n} \phi_k (1-\lambda+k\lambda)A_k \right].
\end{aligned}
\end{equation}
In order to prove that $|a_n| \le A_n$ $(n \in \mathbb{N})$, we use the principle of mathematical induction. Note that:
\begin{equation*}
|a_1| \le A_1 = \frac{\cos\theta(1+\gamma(1-2\lambda))(1-2\lambda)}{(1-\lambda)\phi_1}.
\end{equation*}
Therefore, assume that $|a_k| \le A_k$ for $k = 1, 2, \dots, n; \; n \in \mathbb{N}$. Combining \eqref{eq:bound_n} and \eqref{eq:A_n}, we get:
\begin{equation*}
\begin{aligned}
|a_{n+1}| &\le \frac{2\cos\theta(1+\gamma(1-2\lambda))}{(n+2)(1-\lambda)\phi_{n+1}} \left[ 1-2\lambda + \sum_{k=1}^{n} \phi_k (1-\lambda+k\lambda)|a_k| \right] \\
&\le \frac{2\cos\theta(1+\gamma(1-2\lambda))}{(n+2)(1-\lambda)\phi_{n+1}} \left[ 1-2\lambda + \sum_{k=1}^{n} \phi_k (1-\lambda+k\lambda)A_k \right] \\
&= A_{n+1}.
\end{aligned}
\end{equation*}
Hence, by the principle of mathematical induction, we have $|a_n| \le A_n$ for all $n \in \mathbb{N}$ as desired.\\
Employing the lemma \ref{lem:sequence_An} we get the desired result, i.e coefficient bound for $|a_n|$ for the class of function $f \in \Sigma(\theta, \lambda, \gamma).$
\end{proof}

\section{Radii of Starlikeness and Convexity as applications}

Specifically, we investigate the radii of meromorphic starlikeness and meromorphic convexity of order $\rho$ ($0 \le \rho < 1$) for the integral operator $W_{\alpha, \beta}f(z)$. The following theorem successfully establishes the maximal disks $|z| < r$ inside which the operator maps to domains that are strictly starlike and convex. The proofs inherently rely on the application of the absolute coefficient bounds $\{A_n\}_{n=1}^{\infty}$ derived earlier in Lemma \ref{lem:sequence_An}.

Before stating the main theorem of this section, we recall the foundational definitions and sufficient conditions for meromorphic starlikeness and convexity.

\begin{definition}\cite{mogra}\label{def:starlike}
A meromorphic function $F(z)$ of the form $F(z) = \frac{1}{z} + \sum_{n=1}^{\infty} c_n z^n$ is said to be meromorphically starlike of order $\rho$ ($0 \le \rho < 1$) in the disk $|z| < r$ if it satisfies the analytic condition:
\begin{equation*}
-\Re \left( \frac{zF'(z)}{F(z)} \right) > \rho, \quad (0<|z| < r).
\end{equation*}
A well-known sufficient description for a function to satisfy this property is:
\begin{equation}\label{eq:star_sufficient}
\left| \frac{zF'(z)}{F(z)} + 1 \right| \le 1 - \rho, \quad (0<|z| < r).
\end{equation}
\end{definition}

\begin{definition}\cite{mogra}\label{def:convex}
A meromorphic function $F(z)$ of the form $F(z) = \frac{1}{z} + \sum_{n=1}^{\infty} c_n z^n$ is said to be meromorphically convex of order $\rho$ ($0 \le \rho < 1$) in the disk $0<|z| < r$ if it satisfies the analytic condition:
\begin{equation*}
-\Re \left( 1 + \frac{zF''(z)}{F'(z)} \right) > \rho, \quad (0<|z| < r).
\end{equation*}
A sufficient charecterization for a function to satisfy this property is:
\begin{equation}\label{eq:conv_sufficient}
\left| \frac{zF''(z)}{F'(z)} + 2 \right| \le 1 - \rho, \quad (|z| < r).
\end{equation}
\end{definition}

Utilizing the sufficient inequalities \eqref{eq:star_sufficient} and \eqref{eq:conv_sufficient}, we now determine the precise radii for the operator $W_{\alpha, \beta}f(z)$.


% ...

\begin{theorem}\label{thm:radii}
Let $f \in \Sigma(\theta, \lambda, \gamma)$ and let $\{A_n\}_{n=1}^{\infty}$ be the sequence of coefficient bounds such that $|a_n| \le A_n$ for all $n \in \mathbb{N}$. Then the operator $W_{\alpha, \beta}f(z)$ is:
\begin{enumerate}
    \item[(i)] Meromorphically starlike of order $\rho$ ($0 \le \rho < 1$) in the disk $|z| < r_1$, where $r_1$ is the largest value satisfying the equation:
    \begin{equation} \label{eq:radius_star}
    \sum_{n=1}^{\infty} \left( \frac{n+2-\rho}{1-\rho} \right) \phi_n A_n r^{n+1} \le 1.
    \end{equation}
    
    \item[(ii)] Meromorphically convex of order $\rho$ ($0 \le \rho < 1$) in the disk $|z| < r_2$, where $r_2$ is the largest value satisfying the equation:
    \begin{equation} \label{eq:radius_conv}
    \sum_{n=1}^{\infty} \frac{n(n+2-\rho)}{1-\rho} \phi_n A_n r^{n+1} \le 1.
    \end{equation}
\end{enumerate}
\end{theorem}

\begin{proof}
Let $H(z) = W_{\alpha, \beta}f(z) = z^{-1} + \sum_{n=1}^{\infty} \phi_n a_n z^n$. 

\textbf{Proof of (i):} 
To prove that $H(z)$ is meromorphically starlike of order $\rho$, it suffices to show that
\begin{equation*}
\left| \frac{zH'(z)}{H(z)} + 1 \right| \le 1 - \rho, \quad \text{for } |z| < r_1.
\end{equation*}
Using the series expansions of $H(z)$ and $zH'(z)$, we obtain
\begin{equation*}
\left| \frac{zH'(z) + H(z)}{H(z)} \right| = \left| \frac{\sum_{n=1}^{\infty} (n+1)\phi_n a_n z^n}{z^{-1} + \sum_{n=1}^{\infty} \phi_n a_n z^n} \right|.
\end{equation*}
Multiplying the numerator and denominator by $z$ yields
\begin{equation*}
\left| \frac{\sum_{n=1}^{\infty} (n+1)\phi_n a_n z^{n+1}}{1 + \sum_{n=1}^{\infty} \phi_n a_n z^{n+1}} \right| \le \frac{\sum_{n=1}^{\infty} (n+1)\phi_n |a_n| |z|^{n+1}}{1 - \sum_{n=1}^{\infty} \phi_n |a_n| |z|^{n+1}}.
\end{equation*}
Setting $|z| = r$, the required condition becomes
\begin{equation*}
\sum_{n=1}^{\infty} (n+1)\phi_n |a_n| r^{n+1} \le (1-\rho) \left( 1 - \sum_{n=1}^{\infty} \phi_n |a_n| r^{n+1} \right).
\end{equation*}
Rearranging the terms, we get
\begin{equation*}
\sum_{n=1}^{\infty} (n+2-\rho)\phi_n |a_n| r^{n+1} \le 1 - \rho.
\end{equation*}
Dividing by $(1-\rho)$ and substituting the established bounds $|a_n| \le A_n$, we obtain equation \eqref{eq:radius_star}.

\vspace{0.3cm}
\textbf{Proof of (ii):}
To prove that $H(z)$ is meromorphically convex of order $\rho$, it suffices to show that
\begin{equation*}
\left| \frac{zH''(z)}{H'(z)} + 2 \right| \le 1 - \rho, \quad \text{for } |z| < r_2.
\end{equation*}
Using the series expansions, we calculate
\begin{equation*}
zH''(z) + 2H'(z) = \sum_{n=1}^{\infty} n(n-1)\phi_n a_n z^{n-1} + 2\sum_{n=1}^{\infty} n\phi_n a_n z^{n-1} = \sum_{n=1}^{\infty} n(n+1)\phi_n a_n z^{n-1}.
\end{equation*}
Thus, we evaluate
\begin{equation*}
\left| \frac{zH''(z) + 2H'(z)}{H'(z)} \right| = \left| \frac{\sum_{n=1}^{\infty} n(n+1)\phi_n a_n z^{n-1}}{-z^{-2} + \sum_{n=1}^{\infty} n\phi_n a_n z^{n-1}} \right|.
\end{equation*}
Multiplying the numerator and denominator by $-z^2$, we have
\begin{equation*}
\le \frac{\sum_{n=1}^{\infty} n(n+1)\phi_n |a_n| |z|^{n+1}}{1 - \sum_{n=1}^{\infty} n\phi_n |a_n| |z|^{n+1}}.
\end{equation*}
Setting $|z| = r$ and forcing this to be bounded by $(1-\rho)$, we obtain
\begin{equation*}
\sum_{n=1}^{\infty} n(n+1)\phi_n |a_n| r^{n+1} \le (1-\rho) \left( 1 - \sum_{n=1}^{\infty} n\phi_n |a_n| r^{n+1} \right).
\end{equation*}
Rearranging the terms yields
\begin{equation*}
\sum_{n=1}^{\infty} n(n+2-\rho)\phi_n |a_n| r^{n+1} \le 1 - \rho.
\end{equation*}
Dividing by $(1-\rho)$ and substituting $|a_n| \le A_n$ establishes equation \eqref{eq:radius_conv}. This completes the proof of Theorem \ref{thm:radii}.
\end{proof}

\vspace{0.5cm}
\begin{example}[Meromorphic Starlikeness for $n=1$]
To illustrate the radius of meromorphic starlikeness $r_1$, let us consider a simplified extremal function within our class where the first term dominates, such that $\phi_1 A_1 = 1$ and all subsequent coefficients are zero. 

Applying equation \eqref{eq:radius_star} for $n=1$, the condition for meromorphic starlikeness of order $\rho$ reduces to:
\begin{equation*}
\left( \frac{1+2-\rho}{1-\rho} \right) r^2 \le 1 \implies r_1(\rho) \leq \sqrt{\frac{1-\rho}{3-\rho}}.
\end{equation*}
This explicitly shows that for standard meromorphic starlikeness ($\rho=0$), the maximal radius is $r_1(0) \leq\frac{1}{\sqrt{3}} \approx 0.577$. As the requested order $\rho$ approaches $1$, the allowable radius shrinks to $0$.
\end{example}

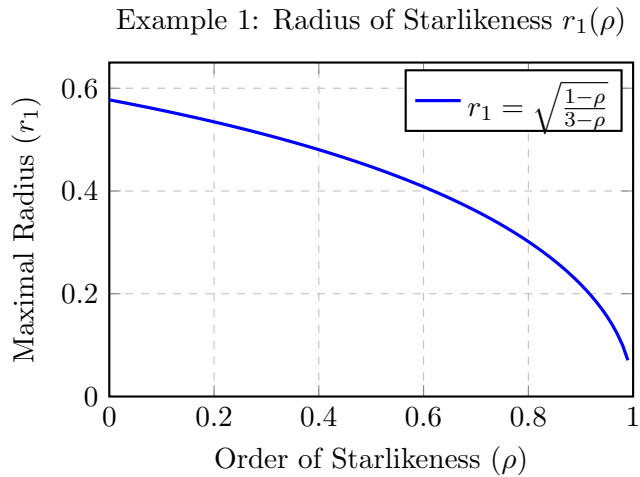
\begin{figure}[htbp]
\centering
\begin{tikzpicture}
\begin{axis}[
    title={Example 1: Radius of Starlikeness $r_1(\rho)$},
    xlabel={Order of Starlikeness ($\rho$)},
    ylabel={Maximal Radius ($r_1$)},
    xmin=0, xmax=1,
    ymin=0, ymax=0.65,
    grid=major,
    grid style={dashed, gray!50},
    domain=0:0.99, 
    samples=100,
    width=8.5cm,
    height=6cm,
    thick
]
\addplot [blue, very thick] {sqrt((1-x)/(3-x))};
\addlegendentry{$r_1 = \sqrt{\frac{1-\rho}{3-\rho}}$}
\end{axis}
\end{tikzpicture}
\caption{The graphical representation of the maximal radius of meromorphic starlikeness $r_1$ as a function of the order $\rho$ for the dominant leading term ($n=1$).}
\label{fig:star_graph}
\end{figure}

\vspace{0.5cm}
\begin{example}[Meromorphic Convexity for $n=2$]
To illustrate the stricter geometric requirement of meromorphic convexity, consider an extremal case where the $n=2$ term is dominant, such that $\phi_2 A_2 = 1$. 

Applying equation \eqref{eq:radius_conv} for $n=2$, the condition for meromorphic convexity of order $\rho$ becomes:
\begin{equation*}
\frac{2(2+2-\rho)}{1-\rho} r^3 \le 1 \implies \left( \frac{8-2\rho}{1-\rho} \right) r^3 \le 1.
\end{equation*}
The radius gives $r_2(\rho) \leq \left( \frac{1-\rho}{8-2\rho} \right)^{1/3}$. For standard meromorphic convexity ($\rho=0$), the maximal radius is $r_2(0) \leq \left(\frac{1}{8}\right)^{1/3} = 0.5$. 
\end{example}

\begin{figure}[htbp]
\centering
\begin{tikzpicture}
\begin{axis}[
    title={Example 2: Radius of Convexity $r_2(\rho)$},
    xlabel={Order of Convexity ($\rho$)},
    ylabel={Maximal Radius ($r_2$)},
    xmin=0, xmax=1,
    ymin=0, ymax=0.55,
    grid=major,
    grid style={dashed, gray!50},
    domain=0:0.99,
    samples=100,
    width=8.5cm,
    height=6cm,
    thick
]
\addplot [red, very thick] {(abs((1-x)/(8-2*x)))^(1/3)};
\addlegendentry{$r_2 = \sqrt[3]{\frac{1-\rho}{8-2\rho}}$}
\end{axis}
\end{tikzpicture}
\caption{The graphical representation of the maximal radius of meromorphic convexity $r_2$ as a function of the order $\rho$ for a higher-order dominant term ($n=2$).}
\label{fig:conv_graph}
\end{figure}
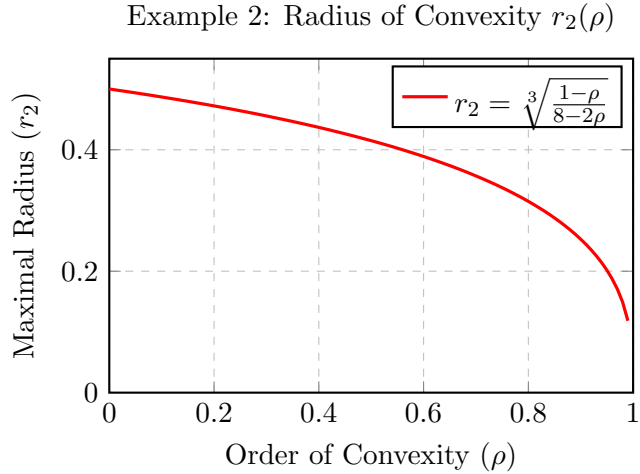

\section{Conclusion}
In the present investigation, we have successfully defined and systematically studied a new subclass $\Sigma(\theta, \lambda, \gamma)$ of meromorphic functions in the punctured unit disk ${D}^*$, utilizing the generalized linear operator $W_{\alpha, \beta}$. The primary achievements of this paper include the derivation of absolute coefficient bounds using rigorous mathematical induction, which formed the essential foundation for the geometric properties explored. We successfully established the integral representation of these functions, along with their necessary and sufficient convolution conditions. Moreover, by applying the derived coefficient estimates, we computed the exact radii of meromorphic starlikeness and meromorphic convexity of order $\rho$ ($0 \le \rho < 1$). 

The findings of this paper not only contribute new theorems to the study of meromorphic functions but also provide a generalized framework. By specializing the parameters (such as $\theta, \lambda,$ or $\gamma$), the results established here can be reduced to unify several well-known subclasses in the existing literature. For future research directions, the methodology and the operator $W_{\alpha, \beta}$ utilized in this study could be extended to investigate bounds for Fekete-Szegö functionals, Toeplitz determinants, Hankel determinants, or adapted to the rapidly developing field of quantum ($q$-) calculus for meromorphic functions.

{}

\end{document}